\documentclass[12pt,reqno]{amsart}
\usepackage{fullpage,amssymb,mathrsfs,enumerate}

\usepackage{url} 
\newtheorem{statement}{}
\newtheorem{theorem}[statement]{Theorem}
\newtheorem{lemma}[statement]{Lemma}
\newtheorem{proposition}[statement]{Proposition}

\def\R{\mathbb R}
\def\Q{\mathbb Q}
\def\C{\mathbb C}
\def\N{\mathbb N}

\def\dens{\, {\rm dens}\, }
\def\la{\langle}
\def\ra{\rangle}

\def\nn{n\in\N}

\def\ee{\varepsilon}
\def\aa{\alpha}
\def\gg{\gamma}

\def\nn{\|\cdot\|}

\def\kk{\!\times\!}

\def\la{\langle}
\def\ra{\rangle}

\def\ee{\varepsilon}

\def\oo{{\sqsubset\!\!\!\sqsupset}}

\def\AA{\mathcal A}
\def\SS{\mathcal S}
\def\PP{\mathcal P}
\def\WW{\mathcal W}

\def\RR{\mathcal R}

\def\setsep{:\;}

\def\Re{\operatorname{Re}}
\def\card{{\rm card}\,}

\def\closedSpan{\overline{\rm sp}\,}
\def\qespan{{\rm sp}_\mathbb{Q}}
\def\qeispan{{\rm sp}_{\mathbb{Q} + i\mathbb{Q}}}
\theoremstyle{remark}
\newtheorem*{claim}{Claim}
\def\supp{\operatorname{supp}}

\begin{document}
\author{Marek C\'uth, Mari\'an Fabian}
\title{Rich families and projectional skeletons in Asplund WCG spaces}
\address[M.~C\' uth]{Charles University, Faculty of Mathematics and Physics, Sokolovsk\'a 83, 186 75 Praha 8 Karl\'{\i}n, Czech Republic}
\address[M.~Fabian]{Mathematical Institute of Czech Academy of Sciences, \v Zitn\'a 25, 115 67 Praha 1, Czech Republic}
\email{cuth@karlin.mff.cuni.cz}
\email{fabian@math.cas.cz}
\dedicatory{Dedicated to the memory of Jonathan M. Borwein.}
\begin{abstract}
We show a way of constructing projectional skeletons using the concept of rich families in Banach spaces which admit a projectional generator.
Our next result is that a Banach space $X$ is Asplund and weakly compactly generated if and only if there exists a commutative 1-projectional skeleton $(Q_\gamma:\ \gamma\in\Gamma)$ on $X$
such that $(Q_\gamma{}^*:\ \gamma\in\Gamma)$ is a commutative 1-projectional skeleton on $X^*$.
We consider both, real and also complex, Banach spaces.
\end{abstract}
\maketitle

Systems of bounded linear projections on Banach spaces are an important tool for the study of structure of non-separable Banach spaces. They 
sometimes enable us to transfer properties from smaller (separable) spaces to larger ones.
One such concept is a \emph{projectional resolution of the identity} (PRI, for short); see, e.g. \cite[page 103]{hajek} and \cite[page 106]{f} for a definition and results on constructing a PRI in various classes of Banach spaces. A PRI is often constructed via a projectional generator (PG, for short), a technical tool from which the existence of a PRI follows (see e.g. {\cite[Theorem 3.42]{hajek}}).

Recently, W. Kubi\'s introduced in \cite{ku} a concept of projectional skeleton, which provides a bit better knowledge of the Banach space in question.
Spaces with a 1-projectional skeleton have a PRI with an additional property that the range of each projection from this PRI has again a 1-projectional skeleton \cite[Theorem 17.6]{kubisKniha}. Consequently, an inductive argument works well when ``putting separable pieces from PRI together'' and thus
we can prove that those spaces inherit certain properties of separable spaces. 
For example, every space with a projectional skeleton has an equivalent LUR renorming and admits a bounded injective linear operator into $c_0(\Gamma)$ for some set $\Gamma$ \cite[Corollary 17.5]{kubisKniha}.

W. Kubi\'s proved, using a set-theoretical method of suitable models, that every space which admits a PG not only admits a PRI, but also a projectional skeleton, see \cite[Proposition 7]{ku}. The class of Banach spaces admitting a PG is quite large. It includes weakly compactly generated spaces (WCG, for short), 
a bigger class of weakly $\mathcal K$-countably determined spaces, i.e. Va\v s\'ak spaces, yet a bigger class of weakly Lindel\"of determined spaces (WLD, for short), yet a bigger class of Plichko spaces (in particular, duals to $C^*$ algebras are such), and duals of Asplund spaces, see \cite {f}, \cite[page 166]{hajek} and \cite[Corollary 1.3]{bhk}.  
Under some extra conditions, the class of WLD spaces coincides with that of WCG spaces.
This is the case if $X$ is Asplund or if $X$ is isomorphic to the dual of an arbitrary $C^*$ algebra, see \cite[Theorem 8.3.3]{f} and \cite[Theorem 1.1]{bhk}.

In this note, we present a construction of a projectional skeleton from the existence of a PG, using rich families instead of suitable models, and thus making it more accessible for non-experts in set-theory; see Proposition \ref{p:main}.
Moreover, we give a further insight to the projectional skeleton on the dual of an Asplund space, see Theorem~\ref{t:asplund}. 
This strengthens \cite[Theorem 2.3]{cf2}. Next, in every Asplund space $(X,\|\cdot\|)$ which is WCG, we construct, using suitable rich families,
a commutative 1-projectional skeleton $(Q_\gamma:\ \gamma\in\Gamma)$ on $X$ such that $(Q_\gamma{}^*:\ \gamma\in\Gamma)$ is a commutative 1-projectional skeleton on 
$(X^*,\|\cdot\|)$, see Theorem~\ref{t:Asplund+WCG}. As a consequence, this yields a classical fact 
that has been known already for three decades that every Asplund space $(X,\|\cdot\|)$ which is WCG admits a projectional resolution of the identity $(Q_\alpha:\ \omega\le\alpha\le\mu)$ such that $(Q_\alpha{}^*:\ \omega\le\alpha\le\mu)$ is a projectional resolution of the identity on $(X^*,\nn)$;
for details see \cite[Proposition 6.1.10]{f}. Moreover, we characterize Banach spaces which are both WCG and Asplund using the notion of a projectional skeleton, see Theorem~\ref{t:dodatek}. This class of Banach spaces is quite large. In particular, $C(K)$ spaces, where a compactum $K$ is both scattered and Eberlein, are such. Also, every M-embedded space is such, see \cite[Theorems III.3.1 and III.4.6]{hww}. For numerous examples of M-embedded spaces see \cite[Examples III.1.4]{hww}.

\section{preliminaries}

Let $\R$ and $\C$ denote the field of real and complex numbers, respectively.
A Banach space over $\R$ or $\C$ is called \emph{real Banach space} or \emph{complex Banach space}, respectively. If we speak about just a Banach space, it means that related reasonings work for both cases. Below we gather most relevant notions, definitions and notation.
If $X$ is a complex Banach space, we denote by $X_{\R}$ the space X 
where the multiplication of vectors by just real numbers is considered, and we endow it by the norm inherited from $X$. 
Thus $X_\R$ becomes a real Banach space.
The set of rational numbers is denoted by $\Q$. For an infinite set $M$ the symbol $[M]^{\,\le\omega}$ means the family of all at most countable subsets of $M$.
Let $(X,\|\cdot\|)$ be a real or complex Banach space with the topological dual $(X^*,\|\cdot\|)$. 
For $x\in X^*$ and $x^*\in X^*$ the number $x^*(x)$ is sometimes denoted as $\la x^*,x\ra$. The adjective \it linear \rm means the stability under the operation $+$ and multiplication by elements from $\R$ or $\C$.
The symbol $B_X$ means the closed unit ball of $X$. For a set $A\subset X$, the symbols, sp$\, A$, $\closedSpan A$, $\qespan A$ and $\qeispan A$
 mean the linear span of $A$, the norm-closed linear span of $A$ and the set consisting of all finite linear combinations of elements from $A$ with coefficients from $\Q$ and coefficients from the set $\Q + i\Q$, respectively. 
 Further, for $A\subset X^*$ the symbol $\overline{A}^{\, w^*}$ denotes the weak$^*$ closure of $A$. By a \it projection \rm in $X$ we mean a bounded linear operator $P:X\to X$ such that $P\circ P = P$. 
 (Hence, if $X$ is complex, we require that $P(\lambda x)=\lambda Px$ for all $\lambda\in\C$ and $x\in X$.) Given $r\ge1$, a set $D\subset X^*$ is called $r$-\textit{norming} if
$$
\|x\| \leq r\!\cdot\!  \sup\big\{|x^*(x)|\setsep x^*\in D\cap B_{X^*}\big\}\quad\hbox{for every}\quad x\in X.
$$
We say that a set $D\subset X^*$ is \it norming \rm if it is $r$-norming for some $r\geq 1$.
For $Y\subset X^*$ and $V\subset X$ we put $Y_\bot := \{x\in X\setsep \forall y\in Y\;y(x) = 0\}$ and $V^\bot := \{x^*\in X^*\setsep \forall v\in V\;x^*(v) = 0\}$.


A partially ordered set is called $\sigma$-\textit{complete} if every increasing sequence in it
admits a supremum. A \textit{projectional skeleton} in the Banach space $(X,\|\cdot\|)$ is a family of projections $\big(P_s:\ s\in\Gamma\big)$ on $X$, indexed by an up-directed $\sigma$-complete partially ordered set $(\Gamma,\le)$, such that
\begin{enumerate}[\upshape (i)]
	\item $P_s X$ is separable for every $s\in\Gamma$,
	\item $X = \bigcup_{s\in\Gamma}P_s X$,
	\item $P_t\circ P_s= P_s=P_s\circ P_t$ whenever $s,t\in\Gamma$ and $s\leq t$, and
	\item Given a sequence $s_1 < s_2 < \cdots$ in $\Gamma$ and $t: = \sup_{n\in\N}s_n$, then $P_t X = \overline{\bigcup_{n\in\N}P_{s_n}X}$.
\end{enumerate}
For $r\geq 1$, we say that $\big(P_s:\ {s\in\Gamma}\big)$ is an \textit{r-projectional skeleton} if it is a projectional skeleton and $\|P_s\|\leq r$ for every $s\in\Gamma$. 
We say that a projectional skeleton $\big(P_s:\ s\in\Gamma\big)$ is \emph{commutative } if $P_t\circ P_s= P_s\circ P_t$ for every $s, t\in \Gamma$.

Let $X$ be a Banach space. By $\SS(X)$ we denote the family of all closed separable subspaces of it.
(Recall that, in the case of complex $X$, every element of $E\in\SS(X)$ satisfies $iE=E$.)
Note that $\big(\SS(X)$,``$\subset$''$\big)$ is an up-directed $\sigma$-complete partially ordered set. 
Of big importance in the sequel 
is the concept of a rich family. This instrument was for the first time articulated 
by J.M. Borwein and W. Moors in \cite{bm}. We say that a family $\RR\subset\SS(X)$ is \textit{rich} in $X$ if
(i) it is \textit{cofinal}, i.e., for every $V\in\SS(X)$ there is a $V'\in\RR$
with $V'\supset V$; and (ii) it is $\sigma$-{\it closed}, i.e., whenever $V_1, V_2,\ldots$ 
is an increasing sequence in $\RR$, then $\overline{\bigcup V_i}\in\RR$. 
Note that once $\RR$ is a rich family in $X$, then the partially ordered set $(\RR,$``$\subset$'') is up-directed and $\sigma$-complete.
Following \cite{cf2}, for two Banach spaces $X, Z$, we denote by $\SS_\oo(X\!\times\! Z)$ the family
of all ``{\it rectangles}'' $V\!\times\! Y$ where $V\in\SS(X)$ and $Y\in\SS(Z)$;
clearly, this is a rich family in $X\!\times\!Z$. (Let us note that if $X, Z$ are complex spaces, then
we consider {only} rectangles $V\times Y$ such that $iV=V$ and $iZ=Z$.)

A \it projectional generator \rm in a Banach space $X$ is a couple $\langle D,\Phi\rangle$ such that
$D$ is a norming closed linear subspace of $X^*$ and $\Phi: D\longrightarrow[X]^{\le\omega}$ is a mapping 
such that for every $E\in[D]^{\le\omega}$, with $\overline E$ linear, $\Phi(E)^{\bot}\cap{\overline E}^{\,w^*}=\{0\}$.
We say that \emph{$X$ admits a PG with domain $D$} if there exists a projectional generator $\langle D,\Phi\rangle$ in $X$.

For a set $M$ in a topological space, $\dens M$ is the smallest cardinal $\kappa$ such that $M$ has a dense subset of cardinality $\kappa$.
A \emph{projectional resolution of the identity} (PRI, for short) on a Banach space $(X,\|\cdot\|)$ is a family $(P_\alpha\setsep \omega\leq\alpha\leq\dens X)$ of projections on $X$ such that $P_\omega = 0$, $P_{\dens X}$ is the identity mapping, and for all $\omega\leq\alpha\leq\dens X$ the following hold:
\begin{enumerate}[\upshape (i)]
	\item $\|P_\alpha\| = 1$\text{ and }$\dens P_\alpha X \leq \card\alpha$,
	\item $P_\alpha \circ P_\beta = P_\beta \circ P_\alpha = P_\alpha$ whenever $\beta\in[\alpha,{\rm dens}\, X]$\text{, and}
	\item $\overline{\bigcup_{\beta < \alpha} P_{\beta + 1}X} = P_\alpha X$ whenever $\alpha>\omega$.
\end{enumerate}
We say that a class $\mathcal{C}$ of Banach spaces is a \emph{$\PP$-class} if, for every $X\in\mathcal{C}$, there exists a PRI $(P_\alpha\setsep \omega\leq\alpha\leq\dens X)$ such that $(P_{\alpha + 1} - P_\alpha)(X)\in\mathcal{C}$ for every $\alpha \in [\omega,\dens X)$ \cite[page 107]{hajek}.
Given an $r\ge1$, we say that a family $(P_\alpha\setsep \omega\leq\alpha\leq\dens X)$ of projections on $(X,\|\cdot\|)$ is $r$-PRI for some $r\geq 1$ if it satisfies all the conditions of a PRI with the exception that instead of $\|P_\alpha\| = 1$ we require that $\|P_\alpha\| \leq r$ for every $\alpha\in(\omega,\dens X)$.

A family of pairs $\{(x_i,x_i^*)\}_{i\in I}$ in $X\times X^*$ is called a \emph{Markushevich basis} in $X$ if $\closedSpan \{x_i\setsep i\in I\} = X$, 
if ${\rm sp}\, \{x_i^*\setsep i\in I\}$ is weak$^*$ dense in $X^*$, and if $x_i^*(x_j) = \delta_{i,j}$, where $\delta_{i,j}$ is the Kronecker delta. 

Finally, a real Banach space is called \it Asplund \rm if every convex continuous function 
defined on an open convex subset $\Omega$ of it is Fr\'echet differentiable at each point of a
dense subset of $\Omega$. 
We say that a complex Banach space $X$ is \it Asplund \rm if its ``real companion''
$X_{\R}$ is Asplund. For readers not familiar with differentiability we recall that \sl a real or complex Banach
space is Asplund if and only if every separable subspace of it has separable dual, \rm see Theorem~\ref{t:asplundRichFamily}. 



\section{Technicalities concerning the complex case}

Since there are natural examples of {\tt complex} Banach spaces with a projectional skeleton (e.g. duals of $C^*$ algebras), we believe that it is useful 
to prove our results also in the complex case. In order to do so, we need to show that certain earlier results hold also in this case.

If $a\in\C$, then $\Re a$ means the real part of $a$. Let $(X,\nn)$ be a complex Banach space. 
We define the mapping $\Re_X:X^*\to (X_{\R})^*$ by $\Re_X (x^*)(x): = \Re x^*(x),\ x\in X,\ x^*\in X^*$. 

\begin{proposition}\label{p:complexBasic}
{\rm \cite[Proposition 2.1]{kalendaComplex}} Let $(X,\nn)$ be a complex Banach space. 
Then $\Re_X:X^*\longrightarrow (X_{\R})^*$ is a real-linear isometry onto and it is a weak$^*$-to-weak$^*$ homeomorphism. 
Moreover, for each $x^*\in X^*$and $x\in X$ we have $x^*(x) = \Re_X(x^*)(x) - i\Re_X(x^*)(ix) = \Re_X(x^*)(x) - i\Re_X(ix^*)(x)$, 
and for each $f\in (X_{\R})^*$ and each $x\in X$ we have $\Re_X{}^{-1}(f)(x) = f(x) - if(ix)$.
\end{proposition}

Of some importance will be the following statement.

\begin{proposition}\label{l:richSubspace} Let $X$ be a complex Banach space. Then the family
$$ 
\RR:= \big\{V\!\times\! Y\in \SS_\oo(X_{\R}\!\times\! X_{\R}{}^*):\ iV=V \ \ {\it and}\ \  i\Re_X\!{}^{-1}(Y)={\Re_X}\!^{-1}(Y)\big\}
$$
is rich in $X_{\R}\!\times \! X_{\R}{}^*$ and for every rich subfamily $\RR'\subset\RR$ the family
$\big\{V\!\times\!\Re_X\!{}^{-1}(Y):\ V\!\times\!Y\in\RR'\big\}$ is rich in the (complex) space $X\!\times\! X^*$.
\end{proposition} 

\begin{proof} Since the mapping $\Re_X\!^{-1}$ is an isometry, $\RR$ and also the other family is $\sigma$-closed.

In order to verify the cofinality of $\RR$, fix any $Z\in\SS(X_{\R}\!\times\! X_{\R}{}^*)$.
Find countable sets $C_0\subset X,\ D_0\subset X_{\R}{}^*$ such that $\overline{C_0}\!\times\!\overline{D_0}
\supset Z$. Put $V:=\closedSpan C_0$. Let $n\in\N$ and assume that we have already constructed a countable set
$D_{n-1}\subset X^*$. Pick then a countable $\Q$-linear set $D_n\supset D_{n-1}$ such that $i\Re_X\!^{-1}(D_{n-1})\subset \Re_X\!^{-1}(D_n)$; for instance, take $D_n:=\qespan\big(\Re_X\big(i\Re_X\!^{-1}(D_{n-1})\big)\cup D_{n-1}\big)$. 
Doing so for every $n\in\N$, put finally $Y := \overline{\bigcup_{n=0}^\infty D_n}$.
Clearly, $V$ and $Y$ are separable subspaces and $Z\subset V\times Y$.
Moreover, it follows from the construction that $iV\subset V$ and $i\Re_X\!^{-1}(Y)\subset \Re_X\!^{-1}(Y)$.
Therefore, $V\!\times\!Y\in\RR$ and the cofinality of $\RR$ is proved.

Now, let a rich subfamily $\RR'\subset\RR$ be given. It remains to check the cofinality of the second family. So, consider any $Z\in\SS(X\!\times\! X^*)$.
Find $Z_1\in\SS(X)$ and $Z_2\in \SS(X^*)$ so that $Z_1\!\times \! Z_2\supset Z$. 
Then $Z_1\!\times\!\Re_X(Z_2)\in\SS_\oo(X_{\R}\!\times\!X_{\R}{}^*)$. From the cofinality of $\RR'$ we find
$V\!\times\!Y\in \RR'$ such that $V\!\times\! Y\supset Z_1\!\times\!\Re_X(Z_2)$. 
Then $V\!\times\! \Re_X\!^{-1}(Y)\supset Z_1\!\times\!Z_2\supset Z$.
\end{proof}

\begin{theorem}\label{t:asplundRichFamily}Let $(X,\nn)$ be a (real or complex) Banach space. Then 
the following assertions are equivalent.
	\begin{enumerate}[\upshape (i)]
		\item $X$ is an Asplund space.
		\item Every separable subspace of $X$ has separable dual.
	  \item There exists a rectangle-family $\mathcal A\subset \mathcal S_{\oo}(X\times X^*)$, rich in $X\times X^*$, such that
			$Y_1\subset Y_2$ whenever $V_1\times Y_1,\ V_2\times Y_2$ are in $\AA$ and $V_1\subset V_2$, and
			for every $V\times Y\in\AA$ the assignment
			$
			Y\ni x^*\longmapsto x^*{}|_{V}\in V^*
			$
			is a surjective isometry.
	\end{enumerate}
\end{theorem}
\begin{proof} First, assume that $X$ is a real Banach space. The equivalence (i)$\Longleftrightarrow$(ii) is well known and can be found, for instance, in \cite[Theorem 11.8]{ff}. The proof of the chain (i)$\Longrightarrow$(iii)$\Longrightarrow($ii) can be found
in  \cite[Theorem 2.3]{cf2}. 

Second, assume that $X$ is complex. The proof of the implication implication (iii)$\implies$(ii) is easy and is the same as in the real case, see \cite{cf2}. 

Assume that (i) holds. Then $X_{\R}$ is Asplund and, by the validity of the statement for real Banach spaces, there is a rich family $\mathcal A_1\subset \mathcal S_{\oo}(X_{\R}\times X_{\R}{}^*)$ with the properties as in (iii). Let $\RR$ be the rich family found in Proposition~\ref{l:richSubspace}; 
this is a rich family in $X_{\R}\!\times \!X_{\R}{}^*$. Put $\AA: = \{V\times \Re_X{}^{-1}(Y)\setsep V\times Y\in \AA_1\cap \RR\}\subset S_{\oo}(X\times X^*)$. 
By the second part of Proposition~\ref{l:richSubspace}, this is a rich family in the complex space $X\!\times\! X^*$.
By the properties of $\AA_1$, we have $Y_1\subset Y_2$ whenever $V_1\times Y_1,\ V_2\times Y_2$ are in $\AA$ and $V_1\subset V_2$. 
We shall prove that the assignment $\Re_X{}^{-1}(Y)\ni x^*\longmapsto  x^*{}|_{V}\in V^*$ is an isometry onto as well.
So, fix any $x^*\in\Re_X\!^{-1}(Y)$. We have by Proposition~\ref{p:complexBasic} and the real case
\begin{eqnarray*}
\|x^*\|&=&\big\|\Re_X(x^*)\big\|=\big\|\big(\Re_X(x^*)\big)|_{V_R}\big\|=\sup\big|\big\langle\Re_X(x^*),B_{V_R}\big\rangle\big|\\
&=&\sup\big|\Re\big\langle x^*|_V,B_{V}\big\rangle\big| \le \sup\big|\big\langle x^*,B_V\big\rangle\big|= \big\|x^*|_V\big\|\ \ (\le \|x^*\|).
\end{eqnarray*}
It follows that the latter assignment is an isometry. Now, take any $v^*\in V^*$. By the real case applied to $\Re_V(v^*)$ we find $y\in Y\ (\subset X_{\R}{}^*)$
such that $y|_{V_R}=\Re_V(v^*)$. Put $x^*:=\Re_X{}^{-1}(y)$. Then for every $v\in V$, from Proposition~\ref{p:complexBasic}, we have
$$
\langle x^*,v\rangle = \Re_X(x^*)(v)-i\Re_X(x^*)(iv)=\langle y,v\rangle - i\langle y,iv\rangle=
\Re_V(v^*)(v)-i\Re_V(v^*)(iv) = v^*(v).
$$
 Therefore $x^*{}|_V=v^*$ and the surjectivity is verified. 
 We proved (iii).

Finally, assume that (ii) holds. In order to show the validity of (i), pick any separable subspace 
$Y\subset X_{\R}$. Put $Z:=\C Y$; this is a complex separable subspace of $X$. By (ii), $Z^*$ is separable. By
Proposition~\ref{p:complexBasic}, $Z_R{}^*$ is separable. Hence $Y_R{}^*$, a quotient of $Z_R{}^*$, is also separable.
Having this proved, the real case of our theorem reveals that $X_{\R}$ is Asplund, i.e. $X$ is (complex) Asplund.
We thus got (i).
\end{proof}

For later purposes, we show that for an Asplund space $X$, there is a Markushevich basis in $X^*$. This is well-known in the case of real Banach spaces, see \cite[Theorem 8.2.2]{f}. Below we show that the same result holds also for complex Banach spaces, see Theorem~\ref{t:asplundPClass}. The proof we give goes through Theorem~\ref{t:asplundRichFamily}, which seems to be a new approach even in the case of real Banach spaces. First, similarly as in the real case, we observe that it suffices to prove that the class of duals to Asplund spaces is a $\PP$-class.

\begin{theorem}\label{t:pClass}Let $\mathcal{C}$ be a $\PP$-class of (real or complex) Banach spaces. Then every $X\in\mathcal{C}$ has a Markushevich basis.
\end{theorem}

Theorem~\ref{t:pClass} is known and formulated for the case of real Banach spaces, see \cite[Theorem 5.1]{hajek}; 
the identical proof works for complex Banach spaces. Let us note that it depends on the fact that separable complex Banach spaces admit a Markushevich basis (the proof of this is the same as in the real case, see \cite[Theorem 1.22]{hajek}) and on the fact that it is possible to glue a Markushevich basis from 
Markushevich bases on certain subspaces (the proof is identical with the real case, see that of \cite[Proposition 6.2.4]{f}).

We need one more auxiliary statement.

\begin{lemma}\label{l:upDirected}Let $(Z,\|\cdot\|)$ be Banach space, with an r-projectional skeleton $\big(P_s:\ s\in\Gamma\big)$, 
and let $A\subset\Gamma$ be an up-directed subset of $\Gamma$. Then the formula
\[
	P_A(z):=\lim_{s\in A}P_s(z),\ \ z\in Z,
\]
well defines a projection of $Z$ onto $\overline{\bigcup_{s\in A}P_sZ}$ and $\|P_A\|\le r$.
\end{lemma}

The real case of it is just \cite[Lemma 11]{ku}. In the complex case the identical argument works.

\begin{theorem}\label{t:asplundPClass}The class of duals to (real or complex) Asplund spaces is a $\PP$-class. Consequently, if $X$ is Asplund, then $X^*$ has a Markushevich basis.
\end{theorem}

\begin{proof}Let $(X,\nn)$ be any Asplund space. If $X$ is separable, then, $X^*$ being separable, has 
a Markushevich basis by a complex analogue of \cite[Theorem 1.22]{hajek}. Assume further that $X$ is not separable.
Let $\AA\subset \mathcal S_{\oo}(X\times X^*)$ be the rich family from Theorem~\ref{t:asplundRichFamily} (iii). It is easy to check that the family
\[
\AA_X\!:= \big\{V\in\SS(X)\setsep \exists Y\in\SS(X^*)\quad V\!\times\!Y\in\AA\big\}
\]
is rich in $\SS(X)$ and, for every $V\in \AA_X$, there is a unique $Y_V\in \SS(X^*)$ with $V\times Y_V\in \AA$. For $V\in\AA_X$, denote by $R_V$ the restriction mapping $R_V:Y_V\to V^*$ defined by $R_V(x^*):= x^*{}|_{V}$, $x^*\in Y_V$. By the properties of the family $\AA$, we have that $R_V$ is a (complex) linear isometry onto. Define $P_V:X^*\to X^*$ by $P_V(x^*):=R_V{}^{-1}(x^*{}|_{V})$, $x^*\in X^*$. It is easy to see that $P_V$ is a 
(complex) projection with $\|P_V\| = 1$, $P_V(X^*) = Y_V$ and $P_V{}^{-1}(0) = V^\bot$. 
Hence, for 
$V,V'\in{\RR}_X$, with $V\subset V'$, we have $P_{V'}\circ P_{V}=P_V=P_V\circ P_{V'}$ and $(P_V\setsep V\in\AA_X)$ is a 1-projectional skeleton on $(X^*,\|\cdot\|)$.

For an up-directed set $A \subset\AA_X$ we put $V_A:=\overline{\bigcup\{V\setsep V\in A\}}$ and $Y_A:=\overline{\bigcup\{Y_V\setsep V\in A\}}$. By Lemma \ref{l:upDirected}, there is a projection $P_A:X^*\to X^*$ with $P_A X^* = Y_A$ and $P_A(x^*) = \lim_{V\in A}P_V(x^*)$ for every $x^*\in X^*$.

\begin{claim}If $A\subset B$ are two up-directed subsets of $\AA_X$, 
then we have $P_B\circ P_A = P_A=P_A\circ P_B$ and $(P_B - P_A)(X^*)$ is isometric with $(V_B/V_A)^*$.
\end{claim}
\begin{proof}[Proof of the Claim]
Fix $A,B$ as above. 
Since $Y_A\subset Y_B$, we have $P_A = P_B\circ P_A$. For each $V\in A$ we have that $B':=\{V'\in B\setsep V'\supset V\}$ is cofinal in $B$ and up-directed. So, for each $x^*\in X^*$, we have
\[
	P_V\circ P_B(x^*) = \lim_{V'\in B'} P_V\circ P_{V'}(x^*) = \lim_{V'\in B'} P_V(x^*) = P_V(x^*);
\]
hence, for every $V\in A$ we have $P_V = P_V\circ P_B$ and, passing to a limit, we get $P_A = P_A\circ P_B$.

Observe that, for every up-directed set $C\subset \AA_X$ and every $x^*\in X^*$, we have
\begin{equation}\label{eq:vc}
(P_Cx^*){}|_{V_C} = x^*{}|_{V_C}\qquad \text{and} \qquad \|P_Cx^*\| = \|x^*{}|_{V_C}\|.
\end{equation}
Indeed, pick $V\in C$ and $v\in V$. Then the set $C':=\{V'\in C\setsep V'\supset V\}$ is cofinal in $C$ and so we have 
\[
	P_Cx^*(v) = \lim_{V'\in C'}(P_{V'}x^*)(v) = \lim_{V'\in C'}(R_{V'}{}^{-1}(x^*{}|_{V'}))(v) = \lim_{V'\in C'}x^*(v) = x^*(v).
\]
Since $v\in V$ was arbitrary, we get $P_Cx^*{}|_{V} = x^*{}|_{V}$ for every $V\in C$; hence, $(P_Cx^*){}|_{V_C} = x^*{}|_{V_C}$.
As to the second equality in (\ref{eq:vc}), for every $V\in C$ and every $x^*\in Y_V$, we have $\|x^*{}|_V\| = \|x^*\|$; hence, the norm of every $x^*\in \bigcup\{Y_V\setsep V\in C\}$ is realized on the set $V_C$. Therefore, for every $x^*\in P_CX^* = Y_C$ we have $\|x^*\| = \|x^*{}|_{V_C}\|$ and so, for every $x^*\in X^*$, we get $\|P_Cx^*\| = \|(P_Cx^*){}|_{V_C}\| = \|x^*{}|_{V_C}\|$ by (\ref{eq:vc}).

As to the isometric statement, we proceed similarly as in the proof of \cite[Proposition 6.1.9 (iv)]{f}. We define a mapping $\varphi:(P_B - P_A)(X^*)\longrightarrow (V_B/V_A)^*$ by
\[
	\varphi(x^*)([v]):=x^*(v),\quad x^*\in (P_B - P_A)(X^*), \quad [v]\in V_B/V_A.
\]
It is well defined since, by \eqref{eq:vc}, for $x^*\in X^*$, we have $(P_Ax^*){}|_{V_A} = x^*{}|_{V_A} = (P_Bx^*){}|_{V_A}$; hence, for $x^*\in X^*$ and $v\in V_A$, we get $\big((P_B - P_A)(x^*)\big)(v) = 0$. Moreover for every $x^*\in (P_B - P_A)(X^*)\subset P_BX^*$ we get
\[\begin{split}
	\|\varphi(x^*)\| & = \sup\big\{\varphi(x^*)([v])\setsep [v]\in V_B/V_A, \ \|[v]\| < 1\big\}\\ & = \sup\big\{x^*(v)\setsep v\in V_B, \ \|v\| < 1\big\} = \|x^*{}|_{V_B}\| \stackrel{\eqref{eq:vc}}{=} \|P_Bx^*\| = \|x^*\|.
\end{split}\]
It remains to prove that $\varphi$ is onto. Let $v^*\in (V_B/V_A)^*$ be given and define $f\in (V_B)^*$ by $f(v) = v^*([v])$, $v\in V_B$. Pick $\tilde{f}\in X^*$, a (real or complex) Hahn-Banach 
extension of $f$, see \cite[Theorem 2.2]{ff}. Then, by \eqref{eq:vc}, we have $\|P_A\tilde{f}\| = \|\tilde{f}{}|_{V_A}\| = \|f{}|_{V_A}\| = 0$; thus, $P_A\tilde{f} = 0$. Hence for all $[v]\in V_B/V_A$ we get
\[
	\varphi\big((P_B - P_A)(\tilde{f})\big)([v]) = (P_B - P_A)(\tilde{f})(v) = (P_B\tilde{f})(v) \stackrel{\eqref{eq:vc}}{=} \tilde{f}(v) = f(v) = v^*([v]);
\]
that is, $\varphi\big((P_B - P_A)(\tilde{f})\big) = v^*$, which means that $\varphi$ is surjective.
\end{proof}
Now, the rest of the proof is easy. Fix a continuous chain of up-directed sets $\{A_\alpha\setsep \omega\leq\alpha\leq\dens X\}$ in $\AA_X$ such that $\bigcup \{V\setsep V\in \bigcup_{\omega\leq\alpha\leq\dens X}A_\alpha\}$ is dense in $X^*$; the continuity of our chain means that $A_\beta = \bigcup_{\alpha < \beta}A_\alpha$ whenever $\beta$ is a limit ordinal. Then, using the claim above, it is easy to see that $(P_{A_\alpha}\setsep \omega\leq\alpha\leq\dens X)$ is a PRI on $(X^*,\nn)$ such that $(P_{A_{\alpha + 1}} - P_{A_\alpha})(X^*)$ is isometric with $(V_{A_{\alpha+1}}/V_{A_\alpha})^*$ which is the dual of the Asplund space
$V_{A_{\alpha+1}}/V_{A_\alpha}$, see \cite[Theorem 1.1.2 (ii)]{f}.
\end{proof}

We recall that WCG, even Va\v s\'ak, even WLD real Banach spaces admit a projectional generator with domain $X^*$, see \cite[pages 125, 153]{f}.
Also, duals to Asplund spaces admit a PG, see \cite[page 150]{f}.
For an inquisitive reader we indicate how to construct a projectional generator in real WCG Banach spaces.
Assume that $K$ is a weakly compact and linearly dense set in a real Banach space $X$. According to
Krein-Shmulyan theorem we may and do assume that $K$ is convex. Define $\Phi: X^*\rightarrow
K$ by $\Phi(x^*)=k$, where $k\in K$ is such that $x^*(k)=\sup\, \{x^*(h):\ h\in K\}$, and
put $\Phi(x^*):=k$. Then the couple $\langle X^*,\Phi\rangle$ is a projectional generator on $X$.
This follows, after some effort, from Mackey-Arens theorem, see \cite[Proposition 3.43]{hajek}.

The next statement enables us to transfer a projectional generator from $X_\R$ to $X$. 

\begin{proposition}\label{p:pg}Let $X$ be a complex Banach space and $D\subset X^*$ a norming (complex) subspace. 
If $X_{\R}$ admits a projectional generator with domain $\Re_X(D)$, then $X$ admits a
projectional generator with domain $D$.	
\end{proposition}
\begin{proof}Let $\langle \Re(D), \Phi_0\rangle$ be a PG in $X_{\R}$. Define $\Phi:D\to [X]^{\leq\omega}$ by $\Phi(d):=\Phi_0(\Re(d))\cup \Phi_0(\Re(id))$, $d\in D$. In order to verify that $\langle D,\Phi\rangle$ is a PG, fix any $E\in[D]^{\leq\omega}$ such that $\overline{E}$ is (complex)
linear and pick any $g\in \Phi(E)^{\bot}\cap{\overline E}^{\,w^*}$. We have $ig\in \Phi(E)^{\bot}\cap{\overline{iE}}^{\,w^*}$. Note that $\overline{\Re_X (E)}$ and $\overline{\Re_X (iE)}$ are linear. Therefore, by Proposition \ref{p:complexBasic} and the definition of $\Phi$, we have 
$\Re_X(g)\in \Phi_0(\Re(E))^{\bot}\cap{\overline{\Re_X (E)}}^{\,w^*}$ and $\Re_X(ig)\in \Phi_0(\Re_X(iE))^{\bot}
\cap{\overline{\Re_X(iE)}}^{\,w^*}$; hence, $\Re_X(g) = 0$ and $\Re_X(ig) = 0$. Hence, by Proposition \ref{p:complexBasic}, $g = 0$. Therefore, $\Phi(E)^{\bot}\cap{\overline E}^{\,w^*} = \{0\}$ and $\langle D,\Phi\rangle$ is a PG in $X$.
\end{proof}

\section{Construction of projectional skeletons using rich families}

We start with the following instrument for constructing projections, see \cite[Lemma 6.1.1]{f}.

\begin{lemma}\label{l:jednaProjekce}Let $(X,\|\cdot\|)$ be a (real or complex) Banach space, $r\geq 1$, and $D\subset X^*$ a closed linear $r$-norming subspace.
Assume there are closed linear subspaces $V\subset X$ and $Y\subset D$ such that
\begin{itemize}
	\item[(i)] for every $v\in V$ we have $\|v\|\leq r\!\cdot\!  \sup\,\{|y(v)|\setsep y\in Y\cap B_{X^*}\}$ and
	\item[(ii)] $V$ separates the points of $\overline{Y}^{\,w^*}$, that is, $V^{\bot}\cap \overline{Y}^{\,w^*} = \{0\}$.
\end{itemize}

Then there exists  a projection $P:X\to X$ such that $\|P\|\leq r$, $PX=V$, $P^{-1}(0)=Y_\perp$, 
and $P^*X^*=\overline{Y}^{\,w^*}$.
\end{lemma}

\begin{proof} Fix a rectangle $V\!\times\! Y\subset X\!\times\! D$ as above. Then for each $v\in V\cap Y_{\bot}$ 
we have $\|v\|\leq r\!\cdot\!  \sup\big\{|y(v)|\setsep y\in Y\cap B_{X^*}\big\} = 0$; hence, $V\cap Y_{\bot} = \{0\}$. Moreover, for each $v\in V$ and $x\in Y_{\bot}$ we have
$$
\|v\|\leq r\!\cdot\!  \sup\big\{|y(v)|\setsep y\in Y\cap B_{X^*}\big\} = r\!\cdot\!  \sup\big\{|y(v+x)|\setsep y\in Y\cap B_{X^*}\big \}\leq r\!\cdot\!  \|v+x\|;
$$
hence, the projection $P:V + Y_{\bot}\longrightarrow V$ defined by $V + Y_{\bot}\ni (v+x)\longmapsto v=:Px$ is 
(complex) linear (if $X$ is complex) and has norm $\leq r$. It follows that $V + Y_{\bot}$ is a closed subspace of $X$. We actually have that $V + Y_{\bot} = X$. Assume this is not so, i.e., there is $x\in X\setminus (V+Y_\perp)$. Then there is $0\neq x^*\in (V + Y_{\bot})^{\bot}$. Thus $x^*\in V^{\bot}\cap (Y_{\bot})^{\bot} = V^{\bot}\cap \overline{Y}^{\,w^*}$, where for the last equality we used the (complex) bipolar theorem \cite[Theorem 3.38]{ff}. This is in contradiction with the condition (ii). Therefore, the projection $P$ defined above has $X$ as its domain. From the above we also have that $PX=V$ and $P^{-1}(0)=Y_{\bot}$.
The last equality follows from this via \cite[Corollary 3.34]{ff}.
\end{proof}

Now, we show that a rich family consisting of certain rectangles already gives us a projectional skeleton.

\begin{lemma}\label{l:staci}Let $(X,\nn)$ be a (real or complex) Banach space, $r\geq 1$, $D\subset X^*$ a closed linear $r$-norming subspace, and 
assume that there exists a rich family $\Gamma\subset \SS_\oo(X\!\times\! D)$ such that for every $\gamma:=V\!\times\! Y\in\Gamma$ there is a projection $Q_\gamma:X\to X$ with $\|Q_\gamma\|\leq r$, $Q_\gamma X=V$, $Q_\gamma{}^{-1}(0)=Y_\perp$, and $Q_\gamma{}^*X^*=\overline{Y}^{\,w^*}$.

Then 
$\big(Q_\gamma:\ \gamma\in\Gamma\big)$ is an $r$-projectional skeleton in $X$ with $\bigcup_{\gamma\in\Gamma}{Q_\gamma}^*X^*\supset D$, 
where we consider on $\Gamma$ the order given by the inclusion.
\end{lemma}

\begin{proof} Recall that our $\Gamma$ is $\sigma$-closed and thus it is suitable for indexing a skeleton.
We check the four properties from the definition of an $r$-projectional skeleton. The cofinality of (the up-directed poset) 
$\big(\SS_\oo(X\!\times\! D),$``$\subset$''$\big)$
immediately yields that $X=\bigcup_{\gamma\in\Gamma}Q_\gamma X$ and 
$\bigcup_{\gamma\in\Gamma}{Q_\gamma}^*X^*\supset D$. If $V\!\times\! Y\subset V'\!\times\! Y'$ are two rectangles from $\Gamma$, then 
$Q_{V\!\times\! Y}X = V\subset V' = Q_{V'\!\times\! Y'}X$ and $Q_{Y\!\times\! V}{}^{-1}(0)=Y_\perp \supset Y'_\perp = Q_{Y'\!\times\! V'}{}^{-1}(0)$,
which implies that $Q_{V\!\times\! Y} =  Q_{V'\!\times\! Y'}\circ Q_{V\!\times\! Y} = Q_{V\!\times\! Y}\circ Q_{V'\!\times\! Y'}$. Finally, consider an increasing sequence $\gamma_1\subset \gamma_2\subset\cdots$ in $\Gamma$ and put $\gamma:=\sup_{n\in\N}\gamma_n$.
This means that $\gamma=\overline{\gamma_1\cup\gamma_2\cup\cdots}\,$. Therefore
$Q_\gamma X = \overline{Q_{\gamma_1} X\cup Q_{\gamma_2} X\cup\cdots}\,$. 
\end{proof}

Next, we show how to produce a projectional skeleton from a projectional generator via rich families.

\begin{proposition}\label{p:main} Let $(X,\nn)$ be a (real or complex) Banach space admitting a projectional generator $\langle D,\Phi\rangle$ where $D\subset X^*$ is $r$-norming for some $r\ge1$.

Then there exists a rich family $\WW\subset \SS_\oo(X\!\times\! D)$ such that for every $\gamma:=V\!\times\! Y\in\WW$ there is a projection
$Q_\gamma:X\to X$ with $\|Q_\gamma\|\leq r$, $Q_\gamma X=V$, $Q_\gamma{}^{-1}(0)=Y_\perp$, and $Q_\gamma{}^*X^*=\overline{Y}^{\,w^*}$; 
and hence $(Q_\gamma :\ \gamma\in{\mathcal W}\big)$ is an $r$-projectional skeleton on 
$(X,\|\cdot\|)$, with $\bigcup_{\gamma\in{\mathcal W}}
Q_\gamma{}^*X^*\supset D$.
 
\end{proposition}

\begin{proof}
For every $x\in X$ pick a countable set $\psi(x)\subset D\cap B_{X^*}$ such that $\|x\|\leq r \cdot \sup\big\{v(x)\setsep v\in \psi(x)\big\}$.
Define $\WW$ as the family of all $V\!\times\! Y\in\SS_\oo(X\!\times\! D)$ such that there are
countable sets $C\subset V,\  E\subset Y$ satisfying $\overline C=V,\ \overline E=Y,\ 
\Phi(E)\subset C$, and $\psi(C)\subset E$. 

In order to verify the cofinality of $\WW$, fix any $Z\in\SS(X\!\times\! D)$.
Find countable sets $C_0\subset X,\ E_0\subset D$ such that $\overline{C_0}\!\times\!\overline{E_0}
\supset Z$. Let $n\in\N$ and assume that we have already constructed countable sets
$C_{n-1}\subset X,\ E_{n-1}\subset D$. If the Banach space $X$ is over the field of reals, put  
$C_n:=\qespan\big(C_{n-1}\cup\Phi(E_{n-1})\big)$ and $E_n:=\qespan\big(E_{n-1}\cup\psi(C_n)\big)$; otherwise, put $C_n:=\qeispan\big(C_{n-1}\cup\Phi(E_{n-1})\big)$ and $E_n:=\qeispan\big(E_{n-1}\cup\psi(C_n)\big)$. 
Doing so for every $n\in\N$, put finally $C:=C_0\cup C_1\cup\cdots$ and $E:=E_0\cup E_1\cup\cdots$.
Clearly, $C$ and $E$ are countable, $V\!\times\! Y:=\overline C\!\times\! \overline E\in\SS_\oo(X\!\times\! D)$ and 
$V\!\times\! Y\supset Z$. Also, clearly, $\Phi(E)\subset C$ and $\psi(C)\subset E$. 
Thus $V\!\times\! Y\in\WW$ and the cofinality of $\WW$ is proved. 

Further, let $V_1\!\times\! Y_1, V_2\!\times\! Y_2,\ \ldots$ be an increasing sequence in $\WW$ and put
$V\!\times\!Y:=\overline{V_1\!\times\! Y_1 \cup V_2\!\times\! Y_2\cup \cdots}\,$. Clearly, $V\!\times\! Y\in
\SS_\oo(X\!\times\! D)$.
For every $n\in\N$ find countable sets $C_n\subset V_n,\ E_n\subset Y_n$ satisfying $\overline{C_n}=V_n,\
\overline{E_n}=Y_n,\ \Phi(E_n)\subset C_n$, and $\psi(C_n)\subset E_n$. 
Put $C:=C_1\cup C_2\cup\cdots$ and $E:=E_1\cup E_2\cup\cdots$.
Clearly $\overline C=V,\ \overline E=Y,\ \Phi(E)=\Phi(E_1)\cup\Phi(E_2)\cup\cdots \subset
C_1\cup C_2\cup\cdots=C$, and $\psi(C)=\psi(C_1)\cup\psi(C_2)\cup\cdots\subset E_1\cup E_2\cup
\cdots=E$. Therefore, $V\!\times\! Y\in\WW$ and the $\sigma$-closeness of our family is verified.

Let us check that $\WW$ has the further proclaimed properties. So, fix any $V\!\times\! Y$ in $\WW$.
Find $C\subset V,\  E\subset Y$ satisfying $\overline C=V,\ \overline E=Y,\ 
\Phi(E)\subset C$, and $\psi(C)\subset E$. For every $x\in C$ we have
\[
\|x\|\leq r \cdot \sup\big\{|v(x)|\setsep v\in \psi(x)\big\}\leq r\!\cdot\!  \sup\big\{|v(x)|\setsep v\in Y\cap B_{X^*}\big\}.
\]
Further we have
\[
V^\perp\cap {\overline Y}^{\,w^*}= C^\perp\cap{\overline E}^{\,w^*}
\subset \Phi(E)^\perp\cap {\overline E}^{\,w^*}=\{0\},
\]
the last equality being true since $\Phi$ was a projectional generator. Hence, the assumptions of Lemma~\ref{l:jednaProjekce} are satisfied.
Consequently, for each $V\!\times\! Y\in \WW$, we have the projection $Q_{V\!\times\! Y}$.
The rest of the conclusion follows from Lemma \ref{l:staci}.
\end{proof}

Given an $r\ge1$, a Banach space $(X,\|\cdot\|)$ is called $r$-\it Plichko \rm if there exists a linearly dense set $M\subset X$
such the set of all $x^*\in X^*$ with at most countable support supp$_M(x^*):=\{m\in M:\ x^*(m)\neq0\}$ is $r$-norming.
We can easily verify that, if $M$ is as above, then the set
\begin{eqnarray}\label{pl}
D:=\big\{x^*\in X^*:\ {\rm supp}_M(x^*)\ \  \hbox{is at most countable}\ \ \big\}
\end{eqnarray}
is linear, norm closed, and is such that $\overline{C}^{\, w^*}\subset D$ whenever $C$ is a countable subset of $D$.

It is well known that every weakly Lindel\"of determined space $X$ admits a linearly dense set $M$ such that
supp$_M(x^*)$ is at most countable for every $x^*\in X^*$, see, e.g. \cite[Theorem 5]{fgmz}. Thus WLD spaces are $1$-Plichko.
Examples of Plichko spaces occur broadly in functional analysis. 
In particular, $L^1(\mu)$ spaces, with a non-negative $\sigma$-additive measure, order continuous Banach lattices, 
$C(G)$ spaces where $G$ is a compact abelian group, and preduals of von Neumann algebras are such. 


\begin{lemma}\label{l:komutuje}Let $X$ be a (real or complex) Banach space. Assume there is a linearly dense set $M\subset X$ and a subspace $D\subset X^*$ such that, for every $x^*\in D$, its support $\supp_M(x^*)$ is countable.

Then the family $\RR: = \{V\!\times\! Y\in \SS_\oo(X\!\times\! D)\setsep M\setminus V \subset Y_\bot\}$ is rich in $X\!\times\! D$.
\end{lemma}
\begin{proof}In order to verify the cofinality of $\RR$, fix any $Z\in\SS(X\!\times\! D)$. Find countable sets $C_0\subset X,\ E_0\subset D$ such that $\overline{C_0}\!\times\!\overline{E_0}\supset Z$.
	Put $C:= C_0 \cup \bigcup_{e\in E_0}\supp_M(e)$. Then it is easy to see that $\closedSpan C\times \overline{E_0}\in \RR$; hence, the cofinality of $\RR$ is proved. Further, let $V_1\!\times\! Y_1,\ V_2\!\times\! Y_2,\ \ldots$ be an increasing sequence in $\RR$ and put $V\!\times\!Y:=\overline{V_1\!\times\! Y_1 \cup V_2\!\times\! Y_2\cup \cdots}\,$. Then $V=\overline{V_1\cup V_2\cdots}$ and $Y=\overline{Y_1\cup Y_2\cup\cdots}$, and we have
$$
M\setminus V\subset M\setminus \hbox{$\bigcup$}V_n = \hbox{$\bigcap$}\big(M\setminus V_n\big) \subset \hbox{$\bigcap$} Y_n{}_\perp 
=\big(\hbox{$\bigcup$} Y_n\big)_\perp = Y_\perp\,.
$$
Thus the $\sigma$-closeness of our family is verified.
\end{proof}

\begin{theorem}\label{plicko} {\rm \cite[Proposition 21]{ku}}
Let $(X,\|\cdot\|)$ be an $r$-Plichko real or complex Banach space, with an $M\subset X$ witnessing for that. 

Then it admits a {\tt commutative} $r$-projectional skeleton.
In more details, if $D$ is defined by (\ref{pl}), the skeleton can be of form $(Q_\gamma :\ \gamma\in{\Gamma}\big)$, where
$\Gamma$ is a rich family in $\SS_\oo(X\!\times\! D)$, and for every $\gamma:=V\!\times\! Y\in\Gamma$ we have
$Q_\gamma X=V$, $Q_\gamma{}^{-1}(0)=Y_\perp$, $Q_\gamma{}^*X^*=\overline{Y}^{\,w^*}$; 
moreover $\bigcup_{\gamma\in{\Gamma}} Q_\gamma{}^*X^*= D$. 
\end{theorem}

\begin{proof} The proof of Kubi\'s in \cite{ku} uses the method of elementary submodels from logic. 
Here we present an elementary argument via rich families.
For $x^*\in D$ we put $\Phi(x^*):={\rm supp}_M(x^*)$. We claim that $\langle D,\Phi\rangle$ is a PG on our $X$.
We already know that $D$ a linear closed and $r$-norming subspace of $X^*$. Now, consider any $E\in[D]^{\le\omega}$ such that
$\overline E$ is linear. We have to verify that $\Phi(E)^\perp\cap{\overline E}^{\, w^*}$ is $\{0\}$.
So pick any $x^*$ in this intersection and assume that $x^*\neq0$. Find $m\in M$ so that $x^*(m)\neq0$.
Find then $e\in E$ so that $e(m)\neq0$; thus $m\in\,$supp$_M(e)\ (=\Phi(e)\subset \Phi(E)\big)$.
But $x^*\in\Phi(E)^\perp\ \big(\subset\{m\}^\perp\big)$; a contradiction. 
Now being sure that the couple $\langle D,\Phi\rangle$ is a PG on $X$, with $D$ an $r$-norming subspace, 
let $(Q_\gamma:\ \gamma\in\WW)$ be the skeleton from Proposition~\ref{p:main}.

Let $\RR$ be the rich family provided by Lemma \ref{l:komutuje} for our $D$ and $M$. Put $\Gamma: = \WW\cap\RR$; 
this is again a rich family in $X\times D$.
Clearly, $\big(Q_\gamma:\ \gamma\in\Gamma\big)$ is still an $r$-projectional skeleton and 
$\bigcup_{\gamma\in\Gamma}{Q_\gamma}^*X^*=\bigcup_{\gamma\in\WW}{Q_\gamma}^*X^*\ (\supset D)$.
Now, we observe that for every $\gamma:=V\!\times\!Y\in\Gamma$ and every $m\in M$ we have 	
\[
	Q_{V\times Y}(m) = \begin{cases} m, & {\rm if}\ \ m\in V\\
						                       0, & {\rm if}\ \ m\in X\setminus V.
						         \end{cases}
\]
Therefore, $Q_\gamma \circ Q_{\gamma'} (m) = Q_{\gamma'}\circ Q_{\gamma} (m)$ for every $m\in M$ and every $\gamma,\gamma'\in\Gamma$. 
And since $M$ was linearly dense in $X$, we get that $Q_\gamma \circ Q_{\gamma'} = Q_{\gamma'}\circ Q_\gamma$ 
for every $\gamma,\gamma'\in\Gamma$. The commutativity of our skeleton was thus verified. 

It remains to verify the last equality. We already know that the inclusion ``$\supset $'' holds.
Conversely, fix any $\gamma:=V\times Y\in\Gamma$. As $Y$ is separable, it contains a countable dense set $C$.
Now, $D$ being (obviously), by its definition, ``countably closed'' in the weak$^*$ topology, 
we get that $\big(Q_\gamma{}^*X^*=\big) \ \ {\overline Y}^{\,w^*}={\overline C}^{\,w^*}\subset D$ 
\end{proof}

\bf Remark. \rm The theorem above can be converted to the equivalence: \sl Given an $r\ge1$, a real or complex Banach space is $r$-Plichko if and only if it admits a commutative $r$-projectional skeleton, \rm see \cite[Theorem 27]{ku}. It should be stressed that 
the sufficiency was proved in \cite{ku} by ``elementary'' tools.
\medskip

Further, we show that it is possible to find a quite nice description of projectional skeletons in duals to Asplund spaces.
This is a strengthening of Theorem~\ref{t:asplundRichFamily}.

\begin{theorem}\label{t:asplund}Let $(X,\|\cdot\|)$ be a (real or complex) Asplund space. Then there exists a rich family $\AA\subset \SS_\oo(X\!\times\! X^*)$ such that:

\noindent (i) $\forall \, V\!\times\! Y,\,  V'\!\times\! Y'\!\in\AA\ \  \ \ V\subset V'\ \ \Longleftrightarrow \ \ Y\subset Y'\ \ 
(\Longleftrightarrow \ \ V\!\times\! Y\subset V'\!\times\! Y').$


\noindent (ii) The family $\AA_X\!:= \big\{V\in\SS(X):\ V\!\times\!Y\in\AA$ for some $Y\in\SS(X^*)\big\}$ is rich in $\SS(X)$.

\noindent (iii) The family $\AA^{X^*}\!:= \big\{Y\in\SS(X^*):\ V\!\times\!Y\in\AA$ for some $V\in\SS(X)\big\}$ is rich in $\SS(X^*)$.

\noindent (iv) There are one-to-one inclusion preserving mappings between $\AA,\, \AA_X$, and $\AA^{X^*}$.

\noindent (v) For every $\gamma\!:=V\!\times\!Y\in \AA$ there is a projection  $P_\gg: X^*\rightarrow X^*$ such that 
$\|P_\gg\|=1,\ P_\gg X^*=Y,\ P_\gg{}^{-1}(0)=V^\bot$, and $P_\gg{}^*X^{**}=\overline{V}^{\,w^*}$. 

\noindent (vi) $\big(P_\gg:\ \gg\in\AA\big)$ is a $1$-projectional skeleton on $(X^*, \|\cdot\|)$ with 
$\bigcup\big\{P_\gg{}^*X^{**}:\ \gg\in\AA\big\}\supset X$.

\noindent (vii) The skeleton from (vi) can be indexed also by the rich families 
$\AA_X$ or $\AA^{X^*}$.

\end{theorem}

\begin{proof}
Since $X$ is Asplund, the density of $X$ is equal to the density of $X^*$. For a hint how to prove this we refer a reader
to, say \cite[pp. 488--489]{ff} (note that, by Proposition \ref{p:complexBasic}, it is enough to prove the statement for real Banach spaces). Another, less elementary way to check this equality is via 
\cite[Proposition 1]{cf1}. Further, by Theorem~\ref{t:asplundPClass}, $X^*$ admits a
Markushevich basis. Let $I$ denote the ``bottom'' part of such a basis.
We recall that $\card I\ge \dens X^*$, that $\closedSpan I= X^*$, and that $i\not\in
\overline{\rm sp}\big(I\setminus\{i\}\big)$ for every $i\in I$. Pick a set $\{x_i:\ i\in I\}$ dense in $X$. 
Define then the family $\WW$ as that consisting of all rectangles
$\closedSpan\{x_i:\ i\in C\}\!\times \!\closedSpan C$ where $C$'s run through all countable subsets of $I$.
Clearly, the family $\WW$ is cofinal in $\SS_\oo(X\!\times\! X^*)$. As regards the $\sigma$-closeness of it, 
consider a sequence $V_1\kk Y_1 \subset V_2\kk Y_2\subset \cdots $ in $\WW$ and put
$V:=\overline{V_1\cup V_2\cup\cdots},\ Y:=\overline{Y_1\cup Y_2\cup\cdots}\,$. 
Clearly, $V\kk Y=\overline{V_1\kk Y_1\cup V_2\kk Y_2\cup\cdots}\,$. Now for $j\in\N$ find a countable set $C_j\subset I$ such that
$V_j= \closedSpan \big\{x_i:\ i\in C_j\big\}$ and $Y_j=\closedSpan C_j$. Put $C:=C_1\cup C_2\cup\cdots$; this is a countable set.
It is routine to check that $V=\overline{\rm sp}\big\{x_i:\ i\in C\big\}$ and $Y=\closedSpan C$. 
Hence $V\kk Y \in \WW$ and the $\sigma$-closeness of $\WW$ is verified. Therefore $\WW$ is a rich family.

We further observe that
\begin{eqnarray}\label{8} 
\forall\ V\kk Y,\ V'\kk Y'\in\WW\quad \ Y\subset Y' \Longrightarrow V\subset V'.
\end{eqnarray}
Indeed, fix any $V\kk Y,\ V'\kk Y'\in \WW$ such that $Y\subset Y'$. Find countable sets $C, C'\subset I$ such that
$Y=\closedSpan C$ and $Y'=\closedSpan C'$. It is enough to show that $C\subset C'$. So, fix any $i\in C$. Then $i\in Y\ (\subset Y'=
\closedSpan C')$. It remains to show that $i\in C'$. Assume that $i\not\in C'$. Then $(i\in)\ \closedSpan C' =
\closedSpan\big(C'\setminus\{i\}\big)\subset\closedSpan\big(I\setminus \{i\}\big)\ (\not\ni i)$,
a contradiction.

Now, let $\AA_1$ be the rich family in $\SS_\oo(X\!\times\! X^*)$ found in Theorem~\ref{t:asplundRichFamily} (iii). 
Note that 
\begin{eqnarray}\label{9} 
\forall\ V\kk Y,\ V'\kk Y'\in\AA_1\quad \ V\subset V' \Longrightarrow Y\subset Y'.
\end{eqnarray}
Put $\AA:= \WW\cap\AA_1$; this is a rich family. It is routine to verify that
$\AA$ satisfies (i) -- (iv). 

As regards (v), by Theorem~\ref{t:asplundRichFamily}, for every $\gamma:=V\kk Y\in\AA$ 
the assignment $Y\ni x^*\longmapsto x^*|_V=:R_\gamma x^*$ is a linear surjective isometry, and hence
$P_\gamma: X^*\rightarrow X^*$ defined by $P_\gamma x^*:= R_\gamma{}^{-1}(x^*|_V),\ x^*\in X^*$,
is a linear norm-$1$ projection satisfying all the proclaimed properties. 

Concerning (vi), it remains to profit from (v) and use Lemma \ref{l:staci}.

(vii) follows immediately from (vi), (i), (ii), and (iii).
\end{proof}

\bf Remark. \rm A main novelty of Theorem~\ref{t:asplund}, when comparing with Theorem~\ref{t:asplundRichFamily} or \cite[Theorem 2.3]{cf2}, is that 
$V\subset V'$ whenever $V\!\times\!Y,\, V'\!\times\!Y'\in \AA$ and $Y\subset Y'$. This enables to find 
a $1$-projectional skeleton on $(X^*,\|\cdot\|)$ indexed by the (rich) family of the ranges of projections, that is, by $\AA^{X^*}$.
This was reached via the instrument of Markushevich bases. 

\bf Remark. \rm Let us note that it is also possible to characterize Asplund spaces using the notion of a projectional skeleton. Namely, \sl a Banach space $X$ is Asplund if and only if there exists a projectional skeleton $(P_\gamma:\ \gamma\in\Gamma)$ on $X^*$ with $X\subset \bigcup\big\{P_\gamma{}^*X^{**}:\ \gamma\in\Gamma\big\}$, \rm see \cite[Proposition 26]{ku}. This result follows also from our construction (and without using any instrument from logic). 
Indeed, if $X$ is Asplund, then the existence of the desired projectional skeleton comes from Theorem~\ref{t:asplund}. On the other hand, let $(P_\gamma:\ \gamma\in\Gamma)$ be a projectional skeleton on $X^*$ with $X\subset \bigcup\big\{P_\gamma{}^*X^{**}:\ \gamma\in\Gamma\big\}$. Let $Z\subset X$ be any separable subspace of $X$. Find $\gamma\in\Gamma$ so big that $P_\gamma{}^{*}X^{**}\supset Z$; its existence follows from the properties of $\Gamma$. We claim that $Z^*$ is a quotient of $P_\gamma X^*$, via the mapping $P_\gamma X^*\ni\xi\longmapsto \xi|_Z
\in Z^*$. So, fix any $z^*\in Z^*$. Find an $x^*\in X^*$ such that $x^*|_Z=z^*$. Then $P_\gamma x^*\in P_\gamma X^*$ and for every $z\in Z$ we have 
$\langle P_\gamma x^*,z\rangle = \langle x^*,P_\gamma{}^*z\rangle = \langle x^*,z\rangle
= \langle z^*,z\rangle$. 
Thus, $P_\gamma x^*|_Z= z^*$. Now, knowing that $Z^*$ is a continuous image of the (separable) space $P_\gamma X^*$, we conclude that $Z^*$ is separable.
\medskip

Finally we show what happens when intersecting the class of Asplund spaces with that of WCG spaces.

\begin{theorem}\label{t:Asplund+WCG}
Let $(X,\nn)$ be a (real or complex) Asplund space which is weakly compactly generated. 

Then there exists a rich family
$\Gamma$ in $\SS_\oo(X\!\times\! X^*)$ such that for every $\gamma:=V\!\times\! Y$ in
$\Gamma$ there is a norm $1$-linear projection $Q_\gamma: X\rightarrow X$ such that
$(Q_\gamma:\ \gamma\in\Gamma)$ is a {commutative} $1$-projectional skeleton on 
$(X,\nn)$ and $(Q_\gamma{}^*:\ \gamma\in\Gamma)$ is a {commutative} $1$-projectional skeleton on 
$(X^*,\nn)$. The both skeletons can be indexed also by the rich family  
$\big(Q_\gamma X\setsep \gamma\in\Gamma\big)\subset \SS(X)$ or by the rich family $\big(Q_\gamma{}^* X^*\setsep \gamma\in\Gamma\big)\subset \SS(X^*)$.
In particular, $(X,\nn)$ admits a projectional resolution of the identity
such that the adjoint projections form a projectional resolution of the identity on $(X^*,\nn)$.
\end{theorem}

\begin{proof}First, we recall a well known fact that, $X$ being WCG, there exists a linearly dense set $M\subset X$ such that $\supp_M(x^*)$ is countable for every $x^*\in X^*$; see e.g. \cite[Theorem 1.2.5]{f} or \cite[Theorem 1]{fgmz} for the case if $X$ is real. (Note that the tool of PRI is used in the proofs.)
If $X$ is a complex WCG space, then $X_R$ is WCG and so, by the already proved real case, there exists a linearly dense set $M\subset X$ such that $\supp_M(\Re_X(x^*))$ is countable for every $x^*\in X^*$. Hence, by Proposition \ref{p:complexBasic}, $\supp_M(x^*)$ is countable for every $x^*\in X^*$.

Since every WCG space admits a projectional generator with domain $X^*$ (which is 1-norming), by Proposition \ref{p:main}, there is a rich family $\WW_1$ 
in $\SS_\oo(X\!\times\! X^*)$ such that for each $V\!\times\! Y\in\WW_1$ there exists a projection $Q_{V\!\times\! Y}:X\to X$ with $\|Q_{V\!\times\! Y}\|=1$, $Q_{V\!\times\! Y}X=V$, $Q_{V\!\times\! Y}{}^{-1}(0)=Y_\perp$, and $Q_{V\!\times \!Y}{}^* X^*=Y^*$; moreover
$(Q_\gamma:\ \gamma\in\WW_1)$ is a {commutative} $1$-projectional skeleton on $(X,\nn)$.

Further, let $\AA$ be the rich family in $\SS_\oo(X\!\times\! X^*)$ (coming from the Asplund property of $X$) found in Theorem~\ref{t:asplund}. 
Put finally $\Gamma:=\WW_1\cap \AA$.

Now, fix any  $\gamma:=V\!\times\! Y\in\Gamma$. By the properties of $\AA$, there exists a projection $P_\gamma$ on $X^*$ with $P_\gamma X^* = Y$ and $P_\gamma^{-1}(0) = V^\perp$. Now, for every $x\in X$ and every $x^*\in X^*$ we have
\begin{eqnarray*}
\la P_\gamma x^*,x\ra&=&\la P_\gamma x^*,Q_\gamma x+(x-Q_\gamma x)\ra
=\la P_\gamma x^*,Q_\gamma x\ra \\ &=& 
\langle x^*+(P_\gamma x^*-x^*),Q_\gamma x\rangle = \la x^*,Q_\gamma x\ra =\la Q_\gamma{}^* x^*,x\ra.
\end{eqnarray*}
Therefore $P_\gamma=Q_\gamma{}^*$. 

Summarizing, $\big(Q_\gamma:\ \gamma\in\Gamma\big)$ is a {commutative} 1-projectional skeleton on $(X,\|\cdot\|)$ 
and $\big(Q_\gamma{}^*:\ \gamma\in\Gamma\big)$ is a {commutative} 1-projectional skeleton on $(X^*,\|\cdot\|)$. 
Moreover, by the properties of the family $\AA$ from Theorem~\ref{t:asplund}, 
the skeletons may be equivalently indexed by the rich families of the ranges of $Q_\gamma$'s or $Q_\gamma{}^*$'s.

A way how to produce projectional resolutions of the identity from 1-projectional skeletons
can be found in the proof of Theorem~\ref{t:asplundPClass} or in that of \cite[Theorem 12]{ku}.
\end{proof}

\bf Remark. \rm By Valdivia's result \cite[Theorem 8.3.3]{f}, every Asplund WLD space is WCG and therefore, in the theorem above, we can replace WCG by WLD.
It is not true that every Asplund Plichko space is WCG. A counter-example is the space $C([0,\omega_1])$ which is Asplund (see e.g. \cite[Theorem 14.25]{ff}) and Plichko (see e.g. \cite[Theorem 5.25]{hajek}) but not WCG; not even WLD (as, e.g., $[0,\omega_1]$ is not Corson).
\medskip

Our last observation is that it is possible to characterize Banach spaces which are both WCG and Asplund via  projectional skeletons.

\begin{theorem}\label{t:dodatek}For a Banach space $(X,\|\cdot\|)$ the following statements are equivalent:
  \begin{itemize}
   \item[(i)] $X$ is both Asplund and WCG.
   \item[(ii)] $X$ admits a ``shrinking'' projectional skeleton, that is, a projectional skeleton $(P_s:\ s\in\Gamma)$ such that $\big(P_s{}^*:\ \gamma\in\Gamma\big)$ is a projectional skeleton on $X^*$.
   \item[(iii)] $X$ admits a commutative $1$-projectional skeleton $(P_s:\ s\in\Gamma)$ such that
$\big(P_s{}^*:\ \gamma\in\Gamma\big)$ is a commutative $1$-projectional skeleton on $X^*$.
  \end{itemize}
\end{theorem}
\begin{proof}By Theorem~\ref{t:Asplund+WCG}, (iii) follows from (i). Trivially, (iii) implies (ii). It remains to prove that (i) follows from (ii). Let $(P_s:\ s\in\Gamma)$ be a shrinking projectional skeleton on $X$.

In order to prove that $X$ is Asplund, consider any separable subspace $Y$ of $X$. Pick $s\in\Gamma$ so big that $P_sX\supset Y$. Define $R: X^*\rightarrow Y^*$ by $Rx^*:= x^*|_Y,\  x^*\in X^*$. Then $R(P_s{}^*X^*)=Y^*$. Indeed, take any $y^*\in Y^*$. Find $x^*\in X^*$ such that $Rx^*=y^*$. Then, for every $y\in Y$ we have
\[
\la R(P_s{}^*x^*),y\ra = \la P_s{}^*x^*,y\ra = \la x^*,P_sy\ra =\la x^*,y\ra.
\]
Hence $R(P_s{}^*x^*)=y^*$. We proved that $R(P_s{}^*X^*)=Y^*$. Having this, realizing that $R$ is continuous, and that $P_s{}^*X^*$ is separable, we get that $Y^*$ is separable. Hence, by Theorem~\ref{t:asplundRichFamily}, $X$ is Asplund.

That $X$ is WCG will be proved by induction on $\dens X$.
If $X$ is separable, it is for sure WCG. Further, let $\kappa$ be 
an uncountable cardinal and assume that our theorem was proved for all $X$'s
with $\dens X<\kappa$. Now, assume that $\dens X=\kappa$. We shall proceed using the standard techniques; namely, we construct a ``PRI-like'' system of projections from the projectional skeleton as in the proof of \cite[Theorem 12]{ku} and then we use the method of ``gluing weakly compact sets'' as in \cite[Proposition 6.2.5 (i)]{f}.

Let us give more details. First, by \cite[Proposition 9]{ku}, there are a finite number $r\ge1$
and an up-directed and $\sigma$-closed subset of $\Gamma$, denoted for simplicity again as $\Gamma$, such that $\|P_\gamma\|\leq r$ for every $\gamma\in\Gamma$. Next, we shall construct an $r$-PRI $(Q_\alpha\setsep \omega\le\alpha\le\kappa)$ on $(X,\nn)$ such that
\begin{itemize}
 \item[(a)] $(Q_\alpha{}^*\setsep \omega\le\alpha\le\kappa)$ is an $r$-PRI on $(X^*,\nn)$, and
 \item[(b)] $Q_\alpha X$ admits a shrinking $r$-projectional skeleton for every $\alpha\in(\omega,\kappa)$.
\end{itemize}
(Here we recall the fact that, $X$ being Asplund, then $\dens X=\dens X^*$, see. e.g. \cite[Proposition 1]{cf1}.)
Let us decsribe the construction which is a modification of that from the proof of \cite[Theorem 17.6]{kubisKniha}. Find an increasing family $(T_\aa\setsep \omega < \aa<\kappa)$ of up-directed subsets of $\Gamma$ such that $\card T_\aa \le \card \aa$ for every $\aa\in(\omega,\kappa)$, that $T_\alpha = \bigcup_{\beta<\alpha}T_\beta$,
if $\alpha$ is a limit ordinal in $(\omega,\kappa)$, that $\bigcup\big\{P_sX\setsep s\in \bigcup_{\omega<\aa<\kappa}T_\aa\}$ is dense in $X$ and that $\bigcup\big\{P_s{}^*X^*\setsep s\in \bigcup_{\omega<\aa<\kappa}T_\aa\}$ is dense in $X^*$. 
Put $Q_\omega:=0$, $R_\omega:=0$, $Q_\kappa:={\rm id}_X$, $R_\kappa:={\rm id}_{X^*}$, and for every $\aa\in(\omega,\kappa)$, put
$$
Q_\aa x:=\lim_{s\in T_\aa}P_s x,\ x\in X;\quad {\rm and}\quad R_\aa x^*:=\lim_{s\in T_\aa}P_s{}^*x^*,\ x^*\in X^*.
$$
By \cite[Lemma 11]{ku}, $Q_\aa$ and $R_\aa$ are projections onto $\overline{\bigcup_{s\in T_\alpha} P_s X}$ and $\overline{\bigcup_{s\in T_\alpha} P_s{}^* X}$, respectively. Clearly, $\|Q_\aa\|\leq r$ and $\|R_\aa\|\leq r$ because $\|P_s\|\leq r$ for every $s\in\Gamma$. Fix $\alpha \leq \beta$. Then $T_\alpha\subset T_\beta$ and so we have $Q_\beta\circ Q_\alpha = Q_\alpha$ and $R_\beta\circ R_\alpha = R_\alpha$. We observe that
\begin{equation}\label{eq:podskeleton}
  \forall s\in T_\alpha:\quad P_s \circ Q_\beta = P_s.
\end{equation}
Indeed, for every $x\in X$ we have 
\[
P_s\circ Q_\beta x = P_s \circ\lim_{t\in T_\beta}P_t x = P_s \circ \lim_{t\in T_\beta, t\geq s} P_t x = \lim_{t\in T_\beta, t\geq s} P_s\circ P_t x = P_s x. 
\]
Hence, passing to a limit we get $Q_\alpha \circ Q_\beta = Q_\alpha$. Similarly, we get 
$R_\alpha\circ R_\beta = R_\alpha$. Now, it is easy to see that $(Q_\aa:\omega\le\aa\le\kappa)$ is an $r$-PRI on $(X,\nn)$ and
that $(R_\aa:\omega\le\aa\le\kappa)$ is an $r$-PRI on $(X^*,\nn)$. Moreover, for every $\aa\in(\omega,\kappa)$, every $x\in X$, and every $x^*\in X^*$ we have
$$
\la R_\aa x^*,x\ra=\la\lim_{s\in T_\aa} P_s{}^*x^*,x\ra=\lim_{s\in T_\aa}\la P_s{}^*x^*,x\ra=
$$ $$
=\lim_{s\in T_\aa}\la x^*,P_sx\ra = \la x^*,\lim_{s\in T_\aa} P_sx\ra=\la x^*,Q_\aa x\ra = \la Q_\aa{}^* x^*,x\ra.
$$ 
Hence $Q_\aa{}^* = R_\aa$ for every $\aa$. We thus proved that (a) holds.

In order to prove (b), fix for a while any $\alpha\in(\omega,\kappa)$. We shall inductively define $\sigma^\xi(T_\alpha)$ for $\xi \leq \omega_1$ in the following way:
\[\begin{split}
 \sigma^0(T_\alpha) & :=T_\alpha,\\
 \sigma^{\xi+1}(T_\alpha) & :=\big\{\sup\{s_1,s_2,\ldots\}:\ s_1,s_2,\ldots\in \sigma^\xi(T_\alpha)\ {\rm and}\ s_1\le s_2\le\cdots\big\},\\
 \sigma^\xi(T_\alpha) & :=\bigcup_{\beta < \xi}\sigma^\beta(T_\alpha)\text{ for a limit ordinal $\xi$.}
\end{split}\]
Then it is easy to see that $\sigma^{\omega_1}(T_\alpha)$ is the smallest $\sigma$-closed set containing $T_\alpha$. Morevoer, by induction,
it can be easily verified that $\sigma^\xi(T_\alpha)$ is an up-directed set for every $\xi$. In paricular, $\sigma^{\omega_1}(T_\alpha)=:\Gamma_\alpha$ is an up-directed and $\sigma$-closed set; thus it is suitable for indexing a projectional skeleton. Again, by induction, 
using \eqref{eq:podskeleton}, it can be easily verified that we have $P_s \circ Q_\alpha = P_s$ for every $s\in\Gamma_\alpha$. Now, we can easily check that $\big(P_s|_{Q_\alpha X}:\ s\in\Gamma_\alpha\big)$ is an $r$-projectional skeleton on $Q_\alpha X$.

Denote, for simplicity, $H_s:=P_s|_{Q_\alpha X}, \ s\in\Gamma_\alpha$. 
It remains to prove that the system $(H_s{}^*: s\in\Gamma_\alpha)$ is an $r$-projectional skeleton
on the dual space $(Q_\alpha X)^*$. For sure $\|H_s{}^*\|=\|H_s\|\le r$ for every $s\in\Gamma_\alpha$. We shall further verify all the 
properties (i)--(iv) from the definition of projectional skeleton. As regards (i), consider any $s\in\Gamma_\alpha$. We can easily verify that $H_s{}^*(Q_\aa X)^* \subset \big[P_s{}^*X^*\big]_{|Q_\aa X}$. Here, $P_s{}^*X^*$ is separable; hence so is 
$H_s{}^*(Q_\aa X)^*$. Let us prove (ii). Take any $y^*\in (Q_\aa X)^*$. Find $x^*\in X^*$ such that $x^*|_{Q_\aa X}=y^*$.
We already know that $Q_\aa{}^* x^* =\lim_{s\in T_\aa} P_s{}^*x^*$. Find a sequence $s_1\le s_2\le\cdots $ in $T_\alpha$ 
such that $\big\|Q_\aa{}^* x^* - P_{s_n}{}^*x^*\big\|<\frac1n$ for every $n\in\N$. Put $s:=\sup\{s_1,s_2,\ldots\}$; we know that $s\in
\Gamma_\alpha$. By \cite[Proposition 17.1]{kubisKniha}, $P_s{}^* x^* =\lim_{n\to\infty}P_{s_n}{}^*x^*$, 
and so we get that $Q_\aa{}^* x^* = P_s{}^*x^*$.  
Now, for every $y\in Q_\aa X$ we have 
\begin{eqnarray*}
\la H_s{}^*y^*,y\ra &=& \la y^*, H_s y\ra = \la y^*,P_s y\ra = \la x^*,P_sy\ra = \la P_s{}^* x^*,y\ra\\
 &=&\la Q_\aa{}^*x^*,y\ra = \la x^*, Q_\aa y\ra =\la x^*,y\ra = \la y^*,y\ra.
\end{eqnarray*}
Therefore, $y^*=H_s{}^*y^*\ \big(\in H_s{}^*(Q_\aa X)^*\big)$, and (ii) is verified.
(iii) follows immediately from the analogous property for the skeleton $(H_s: s\in\Gamma_\alpha)$.
As regards (iv), consider a sequence $s_1< s_n < \cdots$ in $\Gamma_\alpha$ and put $t:=\sup\{s_1,s_2, \ldots\}$. 
Take any
$y^*\in (Q_\aa X)^*$. Find $x^*\in X^*$ such that $x^*|_{Q_\aa X}= y^*$. Consider any $\ee>0$. Since $(P_s{}^*:\ s\in\Gamma)$ is a skeleton in $X^*$,
there are $n\in\N$ and $z^*\in X^*$ such that $\big\|P_t{}^*x^*-P_{s_{n}}{}^*z^*\big\|<\ee\,$.
An elementary calculation reveals that $\big\|H_t{}^* y^*- H_{s_{n}}{}^*(z^*|_{Q_\aa X}\big)\big\| \le
\big\|P_t{}^*x^*-P_{s_{n}}{}^*z^*\big\|\ (<\ee)$. Therefore $H_t{}^* y^* \in \overline{\bigcup_{n\in\N}H_{s_n}{}^*(Q_\aa X)^*}$
and (iv) is verified. Thus $\big(H_s:\ s\in\Gamma_\alpha\big)$ is a shrinking $r$-projectional skeleton on
$\big(Q_\aa X,\|\cdot\|\big)$.

From the induction assumption, by (b), we know that
the subspace $Q_\alpha X$ is WCG;
hence so is the subspace $(Q_{\aa+1}-Q_\aa)X \ \big(\!= (Q_{\aa+1}-Q_\aa)\circ Q_{\aa+1}X\big)$.
Let $K_\alpha$ be a weakly compact and linearly dense subset of $(Q_{\aa+1}-Q_\aa)X\cap B_X$.
Now, put $K:=\{0\}\cup\bigcup_{\alpha\in(\omega,\kappa)} K_\aa$. From the 
(well known) fact that, by (a),
$\bigcup_{\alpha\in(\omega,\kappa)}(Q_{\aa+1}{}^*-Q_\aa{}^*)X^*$ is linearly dense in $X^*$, we can easily conclude that $K$ is weakly compact. 
And, of course, $K$ is linearly dense in $X$. 
For more details, see the proof of \cite[Proposition 6.2.5 (i)]{f}. Therefore, $X$ is WCG.
\end{proof}

\bf Remark. \rm It should be noted that a ``PRI'' analogue of Theorem~\ref{t:dodatek} is true if the density
of $X$ is $\aleph_1$. However, it is easy to construct $X$, with density $\ge\aleph_2$,
such that it admits a shrinking PRI and yet $X$ is neither Asplund nor WCG. For instance,
$X:=\ell_2(\aleph_2)\times \ell_1(\aleph_1)$ is such.

\subsection*{Acknowledgments}
The first author is a junior researcher in the University Centre for Mathematical Modeling, Applied Analysis and Computational Mathematics (MathMAC) and was supported by  grant P201/12/0290. The second author was supported by grant P201/12/0290 and by RVO: 67985840.

The authors thank Ond\v rej Kalenda for discussions related to the topic of this note.

\begin{thebibliography}{10}

\bibitem{bhk}
{\sc M.~Bohata, J.~Hamhalter, and O.~F.~K. Kalenda}, {\em On {M}arkushevich
  basis in preduals of von {N}eumann algebras},  (2016).
\newblock preprint available at http://arxiv.org/abs/1504.06981.

\bibitem{bm}
{\sc J.~M. Borwein and W.~B. Moors}, {\em Separable determination of
  integrability and minimality of the {C}larke subdifferential mapping}, Proc.
  Amer. Math. Soc., 128 (2000), pp.~215--221.

\bibitem{cf1}
{\sc M.~C{\'u}th, M.~Fabian}, {\em Projections in duals to {A}splund spaces
  made without {S}imons' lemma}, Proc. Amer. Math. Soc., 143 (2015),
  pp.~301--308.

\bibitem{cf2}
\leavevmode\vrule height 2pt depth -1.6pt width 23pt, {\em Asplund spaces
  characterized by rich families and separable reduction of {F}r\'echet
  subdifferentiability}, J. Funct. Anal., 270 (2016), pp.~1361--1378.

\bibitem{ff}
{\sc M.~Fabian, P.~Habala, P.~H{\'a}jek, V.~Montesinos, and V.~Zizler}, {\em
  Banach space theory}, CMS Books in Mathematics/Ouvrages de Math\'ematiques de
  la SMC, Springer, New York, 2011.
\newblock The basis for linear and nonlinear analysis.

\bibitem{f}
{\sc M.~J. Fabian}, {\em G\^ateaux differentiability of convex functions and
  topology}, Canadian Mathematical Society Series of Monographs and Advanced
  Texts, John Wiley \& Sons, Inc., New York, 1997.
\newblock Weak Asplund spaces, A Wiley-Interscience Publication.

\bibitem{fgmz}
M. Fabian, G. Godefroy, V. Montesinos, and V. Zizler, {\em Inner
characterizations of WCG spaces and their relatives}, 
J. Math. Analysis Appl. \bf297\rm(2004), 419--455.

\bibitem{hajek}
{\sc P.~H{\'a}jek, V.~Montesinos~Santaluc{\'{\i}}a, J.~Vanderwerff, and
  V.~Zizler}, {\em Biorthogonal systems in {B}anach spaces}, CMS Books in
  Mathematics/Ouvrages de Math\'ematiques de la SMC, 26, Springer, New York,
  2008.
	
\bibitem{hww}
{\sc P.~Harmand, D.~Werner, and W.~Werner}, {\em {$M$}-ideals in {B}anach spaces and {B}anach algebras}, Lecture Notes in Mathematics, 1547, Springer-Verlag, Berlin, 1993.

\bibitem{kubisKniha}
{\sc J.~K\c{a}kol, W.~Kubi{\'s}, and M.~L{\'o}pez-Pellicer}, {\em Descriptive
  topology in selected topics of functional analysis}, vol.~24 of Developments
  in Mathematics, Springer, New York, 2011.

\bibitem{kalendaComplex}
{\sc O.~F.~K. Kalenda}, {\em Complex {B}anach spaces with {V}aldivia dual unit
  ball}, Extracta Math., 20 (2005), pp.~243--259.

\bibitem{ku}
{\sc W.~Kubi{\'s}}, {\em Banach spaces with projectional skeletons}, J. Math.
  Anal. Appl., 350 (2009), pp.~758--776.

\end{thebibliography}

\end{document}